\documentclass{article}

\usepackage[english]{babel}
\usepackage[utf8x]{inputenc}
\usepackage[T1]{fontenc}
\usepackage{comment}
\usepackage{amsmath,amssymb,amsthm,mathtools,mathdots,nicefrac,tikz}
\usepackage{algorithm}
\usepackage[noend]{algpseudocode}

\usepackage[a4paper,top=3cm,bottom=2cm,left=3cm,right=3cm,marginparwidth=1.75cm]{geometry}

\usepackage{amsmath}
\usepackage{amsthm}
\usepackage{bbm}
\usepackage{amssymb}
\usepackage{graphicx}
\usepackage{caption}
\usepackage[colorlinks=true, allcolors=blue]{hyperref}

\theoremstyle{definition} 
\newtheorem{theorem}{Theorem}[section]
\newtheorem{construction}[theorem]{Construction}
\newtheorem{corollary}[theorem]{Corollary}
\newtheorem{lemma}[theorem]{Lemma}

\newcommand{\cS}{\mathcal{S}}

\newcommand{\sn}{\mathrm{sn}}

\DeclareMathOperator{\Div}{Div}

\DeclareMathOperator{\val}{val}
\DeclareMathOperator{\tw}{tw}
\DeclareMathOperator{\gon}{gon}

\DeclareMathOperator{\rank}{rank}

\title{Bounds on higher graph gonality}
\author{Lisa Cenek, Lizzie Ferguson, Eyobel Gebre, Cassandra Marcussen, Jason Meintjes,\\ Ralph Morrison, Liz Ostermeyer, and Shefali Ramakrishna}
\date{}

\usepackage{graphicx}

\begin{document}

\maketitle

\begin{abstract}
    We prove new lower and upper bounds on the higher 
    gonalities of finite  graphs. These bounds are generalizations of known upper and lower bounds for first gonality to higher gonalities, including upper bounds on gonality involving independence number, and lower bounds on gonality by scramble number. We apply our bounds to study the computational complexity of computing higher gonalities, proving that it is NP-hard to compute the second gonality of a graph when restricting to multiplicity-free divisors.
\end{abstract}

\section{Introduction}

Divisor theory on finite graphs, pioneered by Baker and Norine in \cite{BAKER2007766}, provides a graph-theoretic analogue to divisor theory on algebraic curves. We can study divisor theory through the language of chip-firing games.
A divisor on a finite graph can be thought of as a placement of an integer number of chips on the vertices of the graph, with negative numbers representing debt.  We then move the chips around using chip-firing moves, wherein a vertex donates a chip along each incident edge. Of particular importance are the \emph{degree} of a divisor, which is the total number of chips; and the \emph{rank} of a divisor, which counts how much added debt the divisor can eliminate via chip-firing moves (regardless of how that debt is placed).

Let \(r\geq 1\) be an integer.  Our focus in this paper is the \(r^{th}\) (divisorial) gonality \(\gon_r(G)\) of a graph \(G\), the minimum degree of a rank \(r\) divisor on \(G\).  In the algebro-geometric world, this invariant gives information about maps from curves to \(r\)-dimensional space. 

Much of the existing literature has focused on the first gonality \(\textrm{gon}_1(G)\) of a graph.  We recall one such theorem here, where \(\alpha(G)\) denotes the {independence number} of \(G\), the largest set of vertices in \(G\) no two of which are adjacent.

\begin{theorem}[\cite{gonality_of_random_graphs}] \label{theorem:independence_first_gon}
If $G$ is a simple graph with $n$ vertices, then $\gon(G) \leq n - \alpha(G)$.
\end{theorem}

To extend coverage to \(r\geq 2\), we prove the following result. Let \(\alpha_r(G)\) be the \(r^{th}\) independence number of a graph \(G\), which is the maximum size of a vertex set such that each vertex is of distance at least \(r\) from each other vertex in the set; let \(\delta(G)\) denote the minimum valence of a vertex of \(G\); and let the girth of a graph be the minimum length of any cycle.

\begin{theorem}\label{higher_gon_theorem}
Let $r \geq 1$, and let $G$ be a graph on \(n\) vertices with $\delta(G) \geq r$ and $\mathrm{girth}(G) > r + 1$. We have $\gon_r(G) \leq n - \alpha_r(G)$.
\end{theorem}

Our result can be viewed as a generalization of Theorem \ref{theorem:independence_first_gon}. In particular, setting \(r=1\), the assumption \(\delta(G)\geq r=1\) follows immediately for connected graphs on more than one vertex; and \(\textrm{girth}(G)> r+1=2\) is equivalent to the assumption that \(G\) is simple.

We leverage this finding to prove the following NP-hardness result for second gonality when restricted to multiplicity-free divisors, which are divisors where every vertex either has 0 or 1 chips placed on it. 

\begin{theorem}\label{theorem:np_hard_mfgon2} 
The second multiplicity-free gonality of a graph is NP-hard to compute, even for bipartite graphs.
\end{theorem}

This result is compelling for two main reasons. First, it demonstrates how our new upper bound on higher gonalities can be used to prove computational complexity results for higher gonalities. Second, the computational complexity result itself gives proof and evidence that higher gonalities are in fact NP-hard, just like first gonality \cite{gijswijt2019computing}.

In practice, it is often more difficult to find lower bounds than upper bounds on \(\textrm{gon}_r(G)\).  For the case of \(\gon_1(G)\), the well-studied graph parameter of treewidth serves as a lower bound \cite{debruyn2014treewidth}; and more recently the invariant of scramble number was introduced to provide a stronger lower bound on first gonality \cite{scramble}.
We present generalizations of these invariants to the \(r^{th}\) treewidth and \(r^{th}\) scramble number. We prove that these provide new lower bounds on the \(r^{th}\) gonality of a graph, although in practice they do not perform as strongly as in the case of \(r=1\).

This paper is organized as follows. In Section \ref{section:background}, we introduce definitions and background material on existing upper and lower bounds on first gonality. We also introduce important terms for our discussion of the complexity of computing second multiplicity-free gonality. In Section \ref{section:upperbounds}, we prove our upper bound results for higher gonalities. In Section \ref{section:mfgonhard}, we use these new upper bounds to prove the NP-hardness of computing second multiplicity-free gonality. In Section \ref{section:lb_scramble}, we generalize scramble number and treewidth, and prove our lower bound results for higher gonalities. 
\medskip

\noindent \textbf{Acknowledgements.}  The authors were supported by Williams College and the SMALL REU, and by the NSF via grants DMS-1659037 and DMS-2011743.
 
\section{Background and definitions} \label{section:background}

Throughout this paper, a graph $G$ is a pair $G = (V, E)$ of a finite vertex set $V=V(G)$ and a finite edge multiset $E=E(G)$. Note that although we allow multiple edges between two vertices, we do not allow edges from a vertex to itself. All graphs in this paper are assumed to be undirected and connected. If $A$ and $B$ are disjoint subsets of $V(G)$, we let $E(A, B) \subseteq E(G)$ denote the multiset of edges between vertices in $A$ and $B$; that is, the multiset of edges that have one vertex in $A$ and the other in $B$.  The \emph{valence} of a vertex \(v\in V(G)\) is the number of edges incident to a vertex \(v\); or equivalently, \(\textrm{val}(v)=|E(\{v\},\{v\}^C)|\). Given \(A\subset V(G)\), we let \(G[A]\) denote the subgraph of \(G\) \emph{induced} by \(A\); that is, the subgraph with vertex set \(A\) including all possible edges from \(E(G)\).  If \(G[A]\) is a connected subgraph, we call \(A\) a \emph{connected subset} of \(V(G)\).  More generally, if \(G[A]\) is a \(k\)-edge-connected graph (one that remains connected even after removing any \(k-1\) edges), we call \(A\) a \emph{\(k\)-edge connected set of vertices}.

\subsection{Divisor theory on graphs}  We recall  several definitions for divisor theory on graphs;  see \cite{sandpiles} for more details.  
Given a graph $G$, a \emph{divisor} $D$ on $G$ is an element of the free abelian group $\Div(G)$ generated by the vertices $V(G)$; as a set, we have
$$\Div(G) = \mathbb{Z} V = \left\{\sum_{v \in V} D(v) \cdot (v) : D(v) \in \mathbb{Z}\right\}.$$
In the language of chip-firing games, a divisor represents a placement of chips on the vertices, namely \(D(v)\) chips on \(v\), where we describe \(v\) as being \emph{in debt} if \(D(v)\leq -1\).

The \emph{degree} of a divisor $D = \sum_{v \in V} D(v) \cdot (v) \in \Div(G)$ is defined as
$$\deg(D) = \sum_{v \in V} D(v).$$
We say \(D\) is \emph{effective} if \(D(v)\geq 0\) for all \(v\); that is, if no vertex is in debt. Denote by $\Div_+(G)$ the set of effective divisors, and by $\Div_+^d(G)$ the set of all effective divisors of degree \(d\).  If \(D\in\Div_+(G)\), we define the \emph{support of \(D\)} to be \[\textrm{supp}(D)=\{v\in V(G)\,:\, D(v)>0\}. \]

We can transform one divisor \(D\) into another \(D'\) via a \emph{chip-firing move at a vertex \(v\)}.  This divisor \(D'\) is defined by
\[D'(w)=\begin{cases}
D(v)-\val(v)&\textrm{ if \(w=v\)}\\
D(w)+|E(\{v\},\{w\})|&\textrm{ if \(v\neq w\).}
\end{cases}\]
In other words, \(D'\) is obtained from \(D\) by moving chips from \(v\) to each of its neighbors \(w\), with one chip sent along each edge.  We say two divisors \(D\) and \(D'\) are \emph{linearly equivalent}, denoted \(D\sim D'\), if there exists a sequence of chip-firing moves transforming \(D\) into \(D'\).  We may also perform \emph{subset firing moves}, wherein all vertices in some subset \(U\subset V(G)\) are fired (in any order). In this case, a vertex \(v\in U\) loses \(\textrm{outdeg}_U(v)\) chips, where \(\textrm{outdeg}_U(v)\) denotes the number of edges in \(E(U,U^C)\) that have \(v\) as an endpoint.

Number the vertices of \(G\) as \(v_1,\ldots,v_n\). Let \(L\) be the \(n\times n\) matrix whose diagonal entries \(L_{ii}=\val(v_i)\), and whose off-diagonal entries are \(L_{ij}=-|E(v,w)|\); this is known as the \emph{Laplacian matrix} of \(G\). Let \(\Delta:\mathbb{Z}^{V(G)}\rightarrow \textrm{Div}(G)\) denote the \emph{Laplace operator}, which is the map induced by \(L\).  If \(D,D'\in\Div(G)\) are equivalent, say with \(D\) being transformed into \(D'\) via a collection of chip-firing moves encoded by \(f\in\mathbb{Z}^{V(G)}\), then \(D'=D-\Delta f\).  In particular, if \(D'\) is obtained from \(D\) by firing all vertices in the set \(U\subseteq V(G)\), then \(D'=D-\Delta\mathbbm{1}_U\), where \(\mathbbm{1}_U(v)=1\) if \(v\in U\), and \(\mathbbm{1}_U(v)=0\) otherwise.

We can associate to each divisor $D$ a \emph{linear system} consisting of all effective divisors that are linearly equivalent to $D$. We denote this linear system as $|D|$, formally defined as
$$|D| := \{D' \in \Div_+ (G) : D' \sim D\}.$$

Of particular importance to our discussion of higher gonalities is the \emph{rank} of a divisor. We define the rank $r(D)$ of a divisor $D$ as follows. First, if the divisor $D$ is not equivalent to any effective divisor, we define $r(D) = -1$. 
Otherwise, we define $$r(D) = \max \{r \in \mathbb{Z}^+ : |D - D'| \neq \emptyset \text{ for all } D' \in \Div_+^r(G)\}. $$
  We note that if \(D\sim D'\), then \(r(D)=r(D')\).  For a positive integer $r$, we  define the \emph{$r^{th}$ gonality of \(G\)} as the minimum degree of a rank \(r\) divisor on \(G\).

We can  describe \(r^{th}\) gonality in terms of a chip-firing game.  The \emph{$r^{th}$ gonality game} on a graph $G$ is played as follows: 
\begin{enumerate}
    \item The first player places $k$ chips on the vertices $V(G)$.
    \item The second player places $r$ units of debt on the vertices.
    \item The first player then attempts to eliminate debt using chip-firing moves.
\end{enumerate}
If debt can be eliminated, the first player wins; otherwise, the second player wins.
The $r^{th}$ gonality is then the minimum value of $k$ such that the first player has a winning strategy.

\subsection{Tools for bounding gonality}

One somewhat trivial upper bound on the $r^{th}$ gonality of a graph \(G\) is  $\gon_r(G) \leq r \cdot \gon_1(G)$. This follows from the fact that, for a divisor \(D\) of nonnegative rank and a positive integer \(k\), the rank of \(kD=D+\cdots+D\) (\(D\) added to itself \(k\) times) is at least \(k\cdot r(D)\).  Thus if \(r(D)=1\) and \(\deg(D)=\gon_1(G)\), then \(r(kD)\geq k\cdot r(D)=k\) and \(\deg(kD)=k\deg(D)=k\gon_1(G)\).

One of the tools that we will use in proving improved upper bounds on higher gonalities is Dhar's Burning Algorithm \cite{dhar}, which can be used to determine if debt can be eliminated on a given graph $G$ with divisor $D$, where all debt is on vertex $q$. We present the \textit{Modified Dhar's Burning Algorithm}, which is a modification of Dhar's Burning Algorithm introduced in \cite{gonseq} that is more suited for the study of higher gonalities. This algorithm is recursive, calling itself with MDBA.

\begin{algorithm}[hbt]
\caption{Modified Dhar's Burning Algorithm MDBA \cite{gonseq}}
\begin{algorithmic}
\State \textbf{Input:} A divisor $D = D^+ - D^-$, where $D^+\geq 0, D^- > 0$.
\State \textbf{Output:} A divisor $D' \in |D|$ satisfying $D' \geq 0$, or NONE if none exists.

\If {$D \geq 0$} 
    \Return 0
\EndIf
\State $W := V(G) \setminus \text{supp}(D^-)$
\While {$W \neq \emptyset$}
    \If {$D(v) < \text{outdeg}_W(v)$ for some $v \in V(G)$}
        \State $W = W \setminus \{v\}$
    \Else
        \State \Return MDBA($D - \Delta \mathbbm{1}_W$) 
    \EndIf
\EndWhile
\Return NONE
\end{algorithmic}
\end{algorithm}

As the name of this algorithm suggests, we can intuitively think of the Modified Dhar's Burning Algorithm in terms of letting fire spread through the graph. Given a divisor $D = D^+ - D^-$, where $D^+\geq 0, D^- > 0$, the algorithm sets all the vertices in $\text{supp}(D^-)$ on fire. Any edge incident to a burning vertex catches on fire. If a vertex has at least as many chips (thought of as firefighters) on it as it has incident burning edges, the vertex is protected from the fire. Otherwise, the vertex burns. The fire spreads as much as possible; that is, until no new vertices catch fire through the burning process. If, at this stage, the whole graph has not burned, then we fire the set of unburned vertices $W$. We then recursively call the algorithm, running the process again if there remains any debt on the graph.  If, at any point, on the other hand the whole graph burns, then debt on the graph cannot be eliminated; if debt can be eliminated, the algorithm will eventually do so.

Dhar's Burning Algorithm is a key component of the proof of our generalized upper bounds for higher gonalities. It also plays an important role in the proof of the NP-hardness of computing the second multiplicity-free gonality of a graph $G$, particularly in computing the second multiplicity-free gonality of a certain type of graph.

Next, of great importance to our upper bounds is the $k$-independence number of a graph $G$. For two vertices \(u,v\in V(G)\), let \(d(u,v)\) denote the distance between \(u\) and \(v\); that is, the length of the smallest number of edges in \(E(G)\) required to form a path from \(u\) to \(v\).
 The \emph{$k$-independence number of \(G\)}, denoted $\alpha_k(G)$, is the maximum cardinality of a vertex set $S \subset V$ where $d(u,v) > k$ for all $u,v \in S$. Note that the traditional independence number $\alpha(G)$ of a graph $G$ equals $\alpha_1(G)$.

It turns out it is NP-hard to compute \(\alpha_2(G)\), even for fairly restricted families of graphs. The following result will be key in proving Theorem \ref{theorem:np_hard_mfgon2}.

\begin{theorem}[Theorem 2.4 in \cite{bipartite_k_indep_2}]
It is NP-hard to compute the $2$-independence number for bipartite graphs.
\end{theorem}

For the purposes of providing lower bounds on gonality, we now recall definitions of treewidth, brambles, and scrambles.

A \emph{bramble} on a graph \(G\) is a set $\mathcal{B} = \{B_1, \dots, B_n\}$ of connected subsets of $V(G)$ with the property that for any $B, B' \in \mathcal{B}$, either $B \cap B' \neq \emptyset$ or $B \cap B' = \emptyset$ and there is an edge in $E(B, B')$. In either case, we say that $B$ and $B'$ \textit{touch}.  The \emph{bramble order} of \(\mathcal{B}\), denoted \(||\mathcal{B}||\) is the minimum size of a hitting set for \(\mathcal{B}\); that is, the minimum size of a set \(A\subset V(G)\) such that \(A\cap B_i\neq \emptyset\) for all \(i\). The \emph{bramble number} $\text{bn}(G)$ of a graph $G$ is then the maximum bramble order of any bramble on $G$:
$$\text{bn}(G) := \max_{\mathcal{B}} \{||\mathcal{B}||\}.$$

We can then define the  \emph{treewidth} of a graph $G$, denoted $\tw(G)$, as one less than the bramble number:
\[\tw(G)=\textrm{bn}(G)-1.\]
We remark that the more standard definition of treewidth is in terms of tree decompositions; this is proved to be equivalent to the bramble-based definition in \cite{st}.

\begin{theorem}[Theorem 1.1 in \cite{debruyn2014treewidth}]
For a graph \(G\),  $\tw(G)\leq \gon_1(G)$.
\end{theorem}

We now move to define scrambles, which are defined in \cite{scramble} as a generalization of brambles.
A \emph{scramble} $\cS = \{E_1, \dots, E_s \}$ on a graph $G$ is a collection of connected, nonempty subsets $E_i$ of $V(G)$. We call each $E_i$ an \textit{egg} of the scramble.  We note that brambles are all scrambles, but not vice versa.

We now define the order of a scramble.  This was originally defined in \cite{scramble}, but we use an equivalent definition from \cite{echavarria2021scramble}. For a scramble \(\mathcal{S}\), a hitting set is as usual, and an \emph{egg-cut} is a collection of edges in \(E(G)\) that, if deleted, separate the graph into multiple components, at least two of which contain an egg.  The \emph{order} $||\cS||$ of a scramble \(\mathcal{S}\)   is then equal to $\text{min}\{h(\cS), e(\cS) \}$, where
\begin{itemize}
    \item $h(\cS)$ is the minimum size of a hitting set for $\cS$, and
    \item $e(\cS)$ is the smallest size of an egg-cut for $\cS$.
\end{itemize}
The \emph{scramble number} of a graph $G$ is the maximum order of any scramble on $G$; that is,
$$\sn(G) = \max_{\cS} \{||\cS||\}.$$

\begin{theorem}[Theorem 1.1 in \cite{scramble}]
For a graph $G$, $\textrm{tw}(G)\leq \sn(G) \leq \gon_1(G)$.
\end{theorem}

In Section \ref{section:lb_scramble}, we present higher order notions of treewidth and scramble number, which we prove provide lower bounds on $\gon_r(G)$.

\section{Upper bounds on higher gonalities} \label{section:upperbounds}

Throughout this section, fix \(r\) to be a positive integer, and \(G\) to be a graph on \(n\) vertices with $\delta(G) \geq r$ and $\mathrm{girth}(G) > r + 1$.  Recall that we wish to prove that \(\gon_r(G)\leq n-\alpha_r(G)\).

To accomplish this, we let \(S\subset V(G)\) be an \(r\)-independent set of size \(\alpha_r(G)\), and let \(D\in \Div(G)\) be defined by \(D(v)=0\) for \(v\in S\), and \(D(v)=1\) for \(v\notin S\).  Our goal is now to show that \(r(D)\geq r\).  Let \(E\in \textrm{Div}_r^+(G)\), and consider running Modified Dhar's Burning Algorithm on \(D-E\).  Let \(C_1,\ldots,C_m\) be the connected components of the portion of \(G\) that is burned before the burning process stabilizes for the first time; we will refer to these as the \emph{flammable components given by \(D-E\)}.

We start with the following lemma.

\begin{lemma}
Let \(G\), \(D\), \(E\), and \(C_1\ldots C_m\) be as above.  Then there exists a divisor \(E'\) with \(\deg(E')=\deg(E)\) such that \(D-E'\) has the same flammable components \(C_1,\ldots, C_k\);  such that for every \(v\) in some flammable component \(C_i\), we have \((D-E')(v)\leq 0\); and such that each flammable component has the same maximum debt in \(D-E\) and \(D-E'\).
\end{lemma}

\begin{proof}
If for every vertex \(v\) in a flammable component we have \((D-E)(v)\leq 0\), we are done.  Among the vertices \(v\) in \(C_1\cup\cdots\cup C_m\) with $(D-E)(v) = 1$, choose \(w\in C_i\) to be the one that is burned earliest in Dhar's algorithm.  Note that  there are (at least) two burning edges \(e_1\) and \(e_2\) incident to \(w\).  Since \(w\) burned before any other \(1\)-chipped vertex, we can back-track along burning edges to find paths starting with each \(e_i\) that lead, along vertices with \(0\) chips, to vertices \(q_1\) and \(q_2\) with negative numbers of chips.  We claim that these paths do not intersect; in particular, \(q_1\neq q_2\).

Suppose that the paths intersect, and thus must give a cycle of vertices with a total of one chip (namely the chip on \(w\)), possibly fewer if \(q_1=q_2\) is part of the cycle.  Any cycle contains at least \(r+2\) vertices; and by the construction of our divisor \(D\), we have \(D(v)=1\) for at least \(r+1\) vertices in this cycle.  Thus \(E(v)\geq 1\) for at least \(r\) vertices in this cycle to reduce those vertices to \(0\) chips.  But then \(D-E\) is an effective divisor, a contradiction.

Without loss of generality, we will assume that \((D-E)(q_1)\leq (D-E)(q_2)\), and replace \(E\) with \(E-(w)+(q_2)\).  We claim that running Modified Dhar's Burning Algorithm on \(D-(E-(w)+(q_2))\) yields the same flammable components \(C_1,\ldots,C_m\).  Indeed, since we may burn edges and vertices in any order, we may let the fire spread from \(q_1\), unobstructed now by \(w\), all the way to \(q_2\), and from there the burning process is identical.  Also note that each flammable component has the same maximum debt as before.  Compared with \(D-E\), the divisor \(D-(E-(w)+(q_2))\) has one fewer vertex with a positive number of chips. Iteratively modifying the divisor \(E\) in this way gives the desired \(E'\).
\end{proof}

This modification process from \(E\) to \(E'\) is illustrated in Figure \ref{fig:higher_gon_lemma}.

\begin{figure}[hbt]
    \centering
    \includegraphics[scale=0.8]{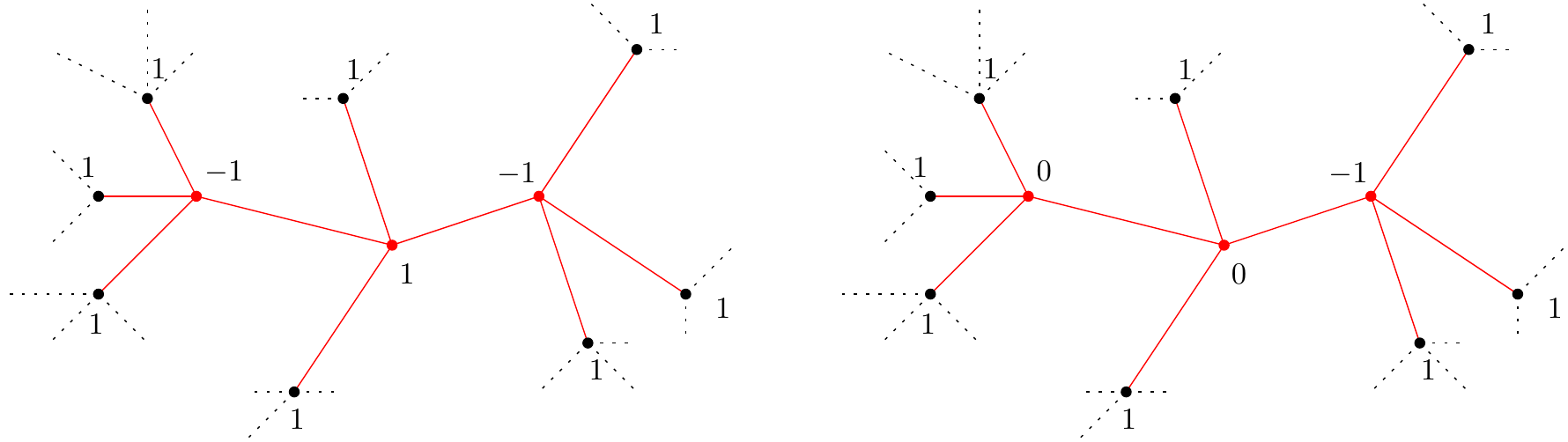}
    \caption{The transformation of \(D-E\) to \(D-E'\) while preserving a flammable component.}
    \label{fig:higher_gon_lemma}
\end{figure}

\begin{lemma} Let \(G\), \(D\), and \(E\) be as above.  A flammable component \(C\) given by \(D-E\) has at most \(r-d+1\) vertices, where \(d\) is the maximum debt present in \(C\) in the divisor \(D-E\); that is, \(-d=\min_{v\in V(C)}(D-E)(v)\).
\end{lemma}

\begin{proof}  By the previous lemma, without loss of generality we may assume that \((D-E)(v)\leq 0\) for all \(v\) in any flammable component (note that it is important that the maximum debt in any flammable component remains unchanged in the previous result).

Let \(C\) be a flammable component given by \(D-E\).  Let \(w\in V(C)\) have the maximum debt \(d\) in \(D-E\). Suppose for the sake of contradiction that \(C\) has more than \(r-d+1\) vertices.  Then we can pick \(r-d+1\) vertices from \(C\), all of which are within distance \(r\) of \(w\).  If  \(D(w)=0\), then every \(v\) within distance \(r\) has \(D(v)=1\), meaning \(E(v)\geq 1\) for the \(r-d+1\) vertices in question.  Combined with \((D-E)(v)=-d\), we have \(\deg(E)\geq r-d+1+d=r+1\), a contradiction.  If \(D(w)=1\), we can still find at least \(r-d\) vertices with one chip within the component \(C\); and the fact that \(E\) must remove an extra chip from \(w\) gives the same contradiction.
\end{proof}

Note that since \(d\geq 1\) for any flammable component, it follows that every flammable component is a tree by our assumption on girth.

\begin{proof}[Proof of Theorem \ref{higher_gon_theorem}]
We may assume without losing any interesting cases that \(G\) is not a tree; indeed, a tree can only satisfy our hypotheses with \(r=1\), and our result is already known to be true in this case.
Let \(\deg(E)=r\) with \(E\geq 0\), and run Modified Dhar's Burning Algorithm on \(D-E\).  Since each flammable component is a tree, the whole graph is not burned; fire all unburned vertices.

Let \(C\) be a flammable component, with maximum debt \(d\) in \(D-E\).  By the previous lemma we know that \(|V(C)|\leq r-d+1\).  Let \(v\in V(C)\).  Since \(\textrm{val}(v)\geq \delta(G)\geq r\),  and since \(v\) has at most \(r-d+1-1=r-d\) neighbors in \(V(C)\), firing the unburned vertices moves at least \(r-(r-d)=d\) chips onto \(v\). Since \((D-E)(v)\geq -d\), debt is eliminated on \(v\).  Thus debt is eliminated on all vertices, and we have that \(r(D)\geq r\).
\end{proof}

We remark that for \(r=2\), our bound is sharp for  \(C_4\), the cycle graph on \(4\) vertices, since \(\alpha_2(C_4)=1\) and \(\gon_2(C_4)=3\) (see for instance \cite[Proposition 3.6]{gonseq}).
For a more interesting example where the bound is sharp, we consider the Crown graph \(Cr_{2n}\) on \(2n\) vertices, which can be defined as a complete bipartite graph on \(n\) and \(n\) vertices with a perfect matching removed.  Note that $C_{2n}$ is $(n-1)$-regular. It is conjectured that the second gonality of the generalized Crown graph is equivalent to our new upper bound; this can be proven for $Cr_{10}$. 

\begin{theorem}
$\gon_2(Cr_{10}) = n - \alpha_2(Cr_{10})$
\end{theorem}

\begin{proof}
Note that $Cr_{10}$ has minimum degree at least \(2\), and girth equal to \(4\).  Moreover, for any vertex there is a unique vertex that is distance greater than \(2\) from it, implying that \(\alpha_2(G)=2\).
By Theorem \ref{higher_gon_theorem}, we know $\rank(D) \geq 2$ so that $\gon_2(G) \leq  8$; an example of a divisor of rank at least \(2\) and degree \(8\) is illustrated in Figure \ref{fig:crown_graph}. To prove equality, we must show that $\gon_2(G)>7$.

\begin{figure}[hbt]
    \centering
    \includegraphics[scale=.40]{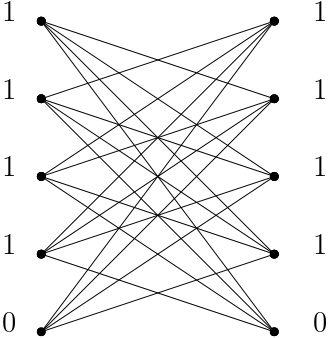}
    \caption{The divisor $D$ of rank at least $2$ guaranteed by Theorem \ref{higher_gon_theorem} on $Cr_{10}$.}
    \label{fig:crown_graph}
\end{figure} 

Now, let \(D\) be an effective divisor of degree \(7\) on  $Cr_{10}$.  Using Theorem \ref{higher_gon_theorem}, without loss of generality we may assume that \(D(v)<\textrm{val}(v)\) for all \(v\), and that if \(v\) and \(w\) are adjacent then we cannot have \(D(v)=D(w)=\textrm{val}(v)\). By construction, there will always be a side (that is, a partite set) of the graph with at least two vertices that have no chips. Let us place $-1$ chips on each of these two vertices, and run Modified Dhar's Burning Algorithm. These two vertices and their outgoing edges burn, meaning that 8 burning edges enter the opposite side of the graph. Since there are only 7 chips distributed throughout the vertices, by the pigeon hole principle, at least one vertex will burn on the other side.

We claim that the fire continues to spread until the whole graph is burned.  If not, when it stabilizes there must be at most \(7\) edges leaving the collection of burned vertices.  Note that in \(Cr_{10}\), a collection of \(3\) vertices can share at most \(2\) edges; a collection of \(4\) vertices at most \(4\) edges; a collection of \(5\), at most \(6\); a collection of \(6\), at most \(8\); and a collection of \(7\), at most \(10\).  If \(u\) vertices are unburned, and \(m_u\) denotes the maximum number of shared edges among those vertices, we have that the number of edges leaving the collection of burned vertices is at least \(4u-2m_u\).  For every value of \(u\) with \(3\leq u\leq 7\), we find  \(4u-2m_u=8>7\), so at least \(8\) vertices must burn.

If \(2\) vertices, say \(v_1\) and \(v_2\), remain unburned, then they must be connected by an edge; otherwise each would have four burning edges incident to it, requiring \(8\) chips. But then each of these two vertices \(v_i\) would have \(D(v_i)\geq 3=\textrm{val}(v_i)-1\), a contradiction.  So at most \(1\) vertex \(v\) remains unburned; but then \(D(v)\geq 4=\textrm{val}(v)\), again a contradiction.  Thus the whole graph burns, and we have that no divisor of degree \(7\) has rank \(2\) or more on \(Cr_{10}\).  We thus conclude that \(\gon_2(Cr_{10})=8\).
\end{proof}

A good direction for future work would be to determine if Theorem \ref{higher_gon_theorem} is sharp for at least one graph when $r \geq 3$.

\section{Computing second multiplicity-free gonality is NP-hard} \label{section:mfgonhard}

In this section we prove our main computational complexity result.  Recall that a effective divisor $D$ for a graph $G$ is {multiplicity-free} if $D(v) \leq 1$ for every vertex $v \in V(G)$; and that the {second multiplicity-free gonality} $\text{mfgon}_2(G)$ of a graph $G$ is the minimum degree of a multiplicity-free divisor of rank $\geq 2$. 

Our proof of Theorem \ref{theorem:np_hard_mfgon2}, that it is NP-hard to compute second multiplicity-free gonality, relies on the following construction. 

\begin{construction}[Bipartite extension of a bipartite graph $G$]
Consider a bipartite graph $G$ with partite sets $B_1$ and $B_2$. Construct a graph $\hat{G}$ as follows: 
\begin{enumerate}
    \item Build a collection of vertices called $A_1$ which has the same number of vertices as $B_2$. Connect the vertices from $A_1$ to the vertices in $B_1$ in the same way as $B_2$ is connected to $B_1$.
    \item Build a collection of vertices called $A_2$ which has the same number of vertices as $B_1$. Connect the vertices from $A_2$ to the vertices in $B_2$ in the same way as $B_2$ is connected to $B_1$.
    \item Add an edge between each pair $(a_1, a_2)$, where $a_1 \in A_1$ and $a_2 \in A_2$.
\end{enumerate}
Denote this new graph constructed as $\hat{G}$, and call it the \textit{bipartite extension} of $G$.  We remark that $\hat{G}$ is indeed a bipartite graph, for instance with partite sets \(A_1\cup B_2\) and \(A_2\cup B_1\).
\end{construction}

Figures \ref{fig:mfgon2_NPhard_G} and  \ref{fig:mfgon2_NPhard_bipartiteextension} give an example of a connected bipartite graph $G$ and its bipartite extension~$\hat{G}$.

\begin{figure}[hbt]
    \centering
    \includegraphics[scale=.40]{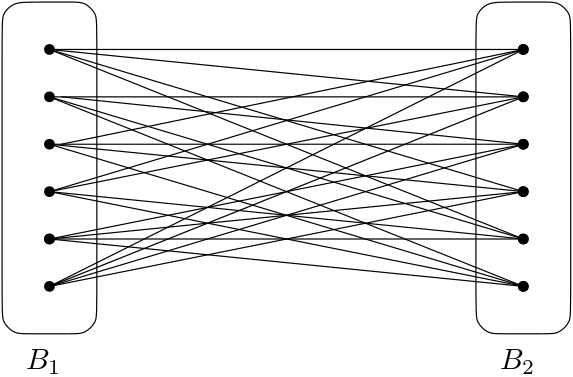}
    \caption{An example of a bipartite graph $G$.}
    \label{fig:mfgon2_NPhard_G}
\end{figure}

\begin{figure}[hbt]
    \centering
    \includegraphics[scale=.50]{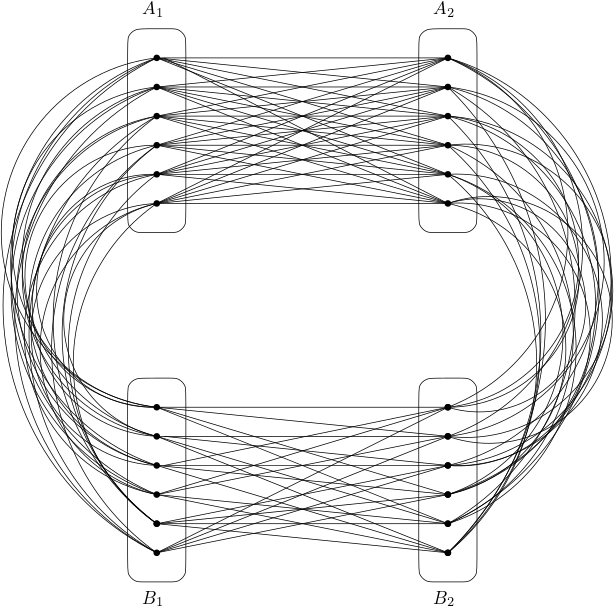}
    \caption{The bipartite extension $\hat{G}$ of the example graph $G$ from Figure \ref{fig:mfgon2_NPhard_G}.}
    \label{fig:mfgon2_NPhard_bipartiteextension}
\end{figure} 

We remark that \(G[A_1\cup A_2]\) is a complete bipartite graph; 
and that the graphs $\hat{G}[B_1\cup B_2]$, $\hat{G}[A_1\cup B_1]$, and $\hat{G}[A_2\cup B_2]$ are  all isomorphic to $G$. 

In the following Lemmas, let $G$ be a connected 4-regular bipartite graph $G$ with partite sets \(B_1\) and \(B_2\), with $|B_1| \geq 2$ and $|B_2| \geq 2$; and let $\hat{G}$ be its bipartite extension.

\begin{lemma} It is NP-hard to compute \(\alpha_2(G)\) for \(4\)-regular bipartite graphs.
\end{lemma}

\begin{proof}
In \cite{bipartite_k_indep_2} the authors provide the following polynomial-time construction:  given a graph \(G\), they produce a regular, bipartite graph \(G'\) such that \(\alpha_1(G)\) and \(\alpha_2(G')\) can be quickly computed from one another.  The valence of the vertices in \(G'\) is \(\max\{4,d\}\), where \(d\) is the maximum degree of a vertex in \(G\).  Thus if \(G\) is a \(3\)-regular graph, \(G'\) is a \(4\)-regular bipartite graph, meaning that we can reduce the problem of computing independence number for \(3\)-regular graphs to the problem of computing \(2\)-independence number for \(4\)-regular bipartite graphs.  Since \(\alpha_1(G)\) is NP-hard to compute even for \(3\)-regular graphs \cite{some_simplified_np_complete}, we have our desired result.
\end{proof}

Our next result connects the \(2\)-independence numbers of \(G\) and \(\hat{G}\).

\begin{lemma} \label{lemma:bipartiteextension_samealpha2}
For a connected, 4-regular bipartite graph $G$ with bipartite extension $\hat{G}$, we have $\alpha_2(\hat{G}) = \alpha_2(G)$.
\end{lemma}

\begin{proof} 
First we will prove that \(\hat{G}\) has a \(2\)-independent set \(S\) of maximum size contained entirely in \(B_1\cup B_2\).
Suppose that \(S'\) is a maximum \(2\)-independent set of \(\hat{G}\) that intersects some \(A_i\); without loss of generality, say \(S'\) contains $v \in A_1$ (or symmetrically in $A_2$).  We will build a \(2\)-independent set \(S\subset B_1\cup B_2\) of the same size as \(S'\).

We first note that $v$ is of distance $\leq 2$ from every vertex in $A_1$, $A_2$ and $B_2$. By construction $v$ is distance 1 away from any $w \in A_2$, and each $w \in A_2$ is of distance 1 away from any $v' \in A_1$, so $v$ is of distance $2$ away from any $v' \in A_1$.  Moreover, if \(w'\in B_2\), then there is some  \(w\in A_2\) adjacent to it (indeed, there are four), so \(d(v,w')=2\).

It follows that every element of \(S'\) besides \(v\) must be in \(B_1\).  Define \(S\) by replacing \(v\) with its corresponding vertex \(v'\) in \(B_2\).  We claim that \(S\) is also a \(2\)-independent set.  Certainly the elements  \(B_1\) are still at distance at least \(3\) from one another.  And by the structure of \(\hat{G}\), the only way for \(v'\in B_2\) to be within distance \(2\) of any vertex of \(B_1\cap S=B_1\cap S'\) is for it to be adjacent to it; but then \(v\) would have been been adjacent to the same vertex, contradicting \(S'\) being \(2\)-independent.  Thus \(S\subset B_1\cup B_2\) is a maximum \(2\)-independent set.

Note that if \(S\subset B_1\cup B_2\) is a \(2\)-independent set in \(\hat{G}\), then it is also a \(2\)-independent of \(G\), meaning \(\alpha_2(G)\geq \alpha_2(\hat{G})\).  Now let \(T\subset V(G)\) be a \(2\)-independent set for \(G\).  We claim that it is also a \(2\)-independent set for \(\hat{G}\).  Let \(v,w\in T\).  If \(v\in B_1\) and \(w\in B_2\), then by \(2\)-independence in \(G\) there is no \(vw\)-path of length at most \(2\) contained in \(\hat{G}[B_1\cup B_2]\); and by the bipartite structure of \(\hat{G}\), any \(vw\)-path that passes outside of \(B_1\cup B_2\) must have length at least \(3\).  If \(v,w\in B_1\), the only possibility for a path of distance at most \(2\) would be in \(\hat{G}[B_1\cup B_2]\) or \(\hat{G}[A_1\cup B_1]\). But both these graphs are isomorphic to \(G\), so such a path cannot exist in either by \(2\)-independence in \(G\). A similar argument holds if \(v,w\in B_2\). Thus any \(2\)-independent set in \(G\) is also \(2\)-independent in \(\hat{G}\).  This gives us \(\alpha_2(\hat{G})\geq \alpha_2(G\)), so we conclude that \(\alpha_2(\hat{G})= \alpha_2(G\)).
\end{proof}

We now compute the second multiplicity-free gonality of \(\hat{G}\).

\begin{lemma}\label{lemma:mfgon2_formula}
Let $G$ be a connected, simple, $4$-regular bipartite graph. Then \[\text{mfgon}_2 (\hat{G}) = |V(\hat{G})| - \alpha_2(\hat{G}).\]
\end{lemma}

\begin{proof}
First we note that \(\delta(\hat{G})>1\); and since \(\hat{G}\) is bipartite, we also have \(\textrm{girth}(\hat{G})>3=2+1\). By Theorem \ref{higher_gon_theorem} with \(r=2\), we know that $\gon_2(\hat{G}) \leq |V(\hat{G})| - \alpha_2(\hat{G})$. Since the divisor from that proof was multiplicity-free, we have $\text{mfgon}_2(\hat{G}) \leq |V(\hat{G})| -  \alpha_2(\hat{G})$.

Now, let \(D\) be a multiplicity-free divisor of degree \(|V(\hat{G})| - \alpha_2(\hat{G}) - 1\); we wish to show it does not have positive rank.  Since the complement of the support of \(D\) has size  \(\alpha_2(\hat{G})+1\), we know the unchipped vertices do not form a \(2\)-independent set, so there must be two vertices \(v\) and \(w\) with \(D(v)=D(w)=0\) and \(d(v,w)\leq 2\).  Either \(d(v,w)=2\) and there exists a vertex \(u\) incident to both \(v\) and to \(w\); or \(d(v,w)=1\), in which case choose \(u\) to be any vertex besides \(w\) incident to \(v\).  Thus \(G[{u,v,w}]\) is a path of length \(3\) with at most one chip on it; consider the divisor \(D-(u)-(v)\), and run Dhar's algorithm on it. The vertices \(u,v,w\) will all burn.  We will show that consequently the whole graph burns.  Note that in order for a vertex to burn, it only needs two incident burning edges.

We claim that if we can show that either \(A_1\) or \(A_2\) burns, then the whole graph will burn.  For if \(A_1\) burns, then \(A_2\) will burn; and, each vertex in \(B_i\) is incident to at least two vertices in \(A_i\).  It is therefore enough to show that two vertices in some \(A_i\) burn, as this is enough to burn the other \(A_j\). As a technicality, we note that $|A_1| = |A_2| = |B_1| = |B_2| \geq 2$ by construction, so there will in fact be at least two vertices in each $A_i$.

We now deal with several cases; we will relabel our vertices to assume that among the three vertices on our path we have \(d(u,w)=2\), so \(v\) is the intermediate vertex.

\textit{Case 1:} \(u,w\) are in the same \(S\) where \(S\in \{A_1,A_2,B_1,B_2\}\). If \(u,w\in A_i\) for some \(i\), we are done, since we have two vertices burning in an \(A_i\).  Otherwise, without loss of generality, we have \(u,w\in B_1\).  Either \(v\in A_1\), in which case we know there exists a vertex \(v'\) in \(B_2\) incident to \(u\) and \(w\); or \(v\in B_2\), and there exists a vertex \(v'\) in \(A_1\) incident to both \(u\) and \(w\).  Either way, we can refer to these vertices as \(x\in A_1\) and \(y\in B_2\), and we have that they both burn since they are both incident to \(u\) and \(w\).  Then any vertex in \(A_2\) incident to \(y\) will burn, since it is also incident to \(x\); there are at least two such vertices in \(A_2\), so we have two vertices burning in \(A_2\) and we are done.

\textit{Case 2:} \(u,w\) are in different sets among the four.  Up to symmetry the only difference is whether there are two \(A\)'s or two \(B\)'s among our three vertices.  So for one subcase, assume \(u\in B_1\), \(v\in A_1\), and \(w\in A_2\).  Every vertex in \(A_1\) incident to \(B_1\) burns, since they also have a burning edge from \(w\); thus at least \(2\) vertices in \(A_1\) burn, and we are done.  For the other subcase, assume \(u\in B_1\), \(v\in B_2\), and \(w\in A_2\).  There are at least two vertices in \(A_1\) incident to both \(u\) and \(w\), so they burn, and we are done.

Since the whole graph burns in both possible cases, we know that \(r(D-(u)-(v))=-1\), implying that \(r(D)\leq 1\).  Since \(D\) was an arbitrary multiplity-free divisor of degree \(|V(\hat{G})| - \alpha_2(\hat{G}) - 1\), we conclude that \(\textrm{mfgon}_2(G)\geq|V(\hat{G})| - \alpha_2(\hat{G}) \).  Combined with our other inequality, we have
\[\text{mfgon}_2 (\hat{G}) = |V(\hat{G})| - \alpha_2(\hat{G}),\]
as claimed.
\end{proof}

We are ready to prove that second multiplicity-free gonality is NP-hard to compute, even for bipartite graphs.

\begin{proof}[Proof of Theorem \ref{theorem:np_hard_mfgon2}]
Let \(G\) be a \(4\)-regular bipartite graph.  First we note that $\hat{G}$ is a bipartite graph of polynomial size in comparison to the size of $G$, as measured in the number of vertices and edges of the graph.

We know from Lemmas \ref{lemma:bipartiteextension_samealpha2} and \ref{lemma:mfgon2_formula} that $\alpha_2(G) = \alpha_2(\hat{G}) = \text{mfgon}_2(\hat{G}) -  |V(\hat{G})| $. Thus given an algorithm to compute \(\text{mfgon}_2\) for bipartite graphs, we obtain an algorithm to compute \(\alpha_2\) for regular bipartite graphs with only polynomial blow-up.  Since this is an NP-hard problem, so too is the computation of \(\text{mfgon}_2\) for bipartite graphs. 
\end{proof}

\section{Lower bounds on higher gonality}\label{section:lb_scramble}

We now present new directions for studying lower bounds on higher gonalities of graphs, through a generalization of scramble number \cite{scramble}.

Given an effective divisor $D$ on a graph $G$, we define the \emph{multi-support} of $D$, denoted $\text{msupp}(D)$, as the multisubset of $V(G)$ where each $v \in V(G)$ appears exactly $D(v)$ times.   For example, if $D = 3v_1 + 2v_2 + 0v_3 + v_4$, then $\text{msupp}(G) = \{v_1, v_1, v_1, v_2, v_2, v_4\}$. The intersection of a multisupport with another set should also be considered as a multiset. If \(A\) and \(B\) are multisets that intersect in a multiset with at least \(k\) elements, we say that they \emph{\(k\)-intersect}.

The following lemma is a generalization of \cite[Lemma 2.2]{debruyn2014treewidth}.  We remark that we consider a graph consisting of a single vertex to be \(r\)-edge connected for every \(r\).

\begin{lemma} \label{lem_egg_containment}
Let \(r\geq 1\), and let \(D,D'\) be effective divisors such that \(D'\) is obtained from \(D\) by firing the set \(U\subset V(G)\). Let \(B\) $\subseteq V(G)$ be such that $G[B]$ is $r$-edge-connected, and suppose that \(B\) $r$-intersects $\text{msupp}(D)$ but does not $r$-intersect $\text{msupp}(D')$.  Then \(B\subseteq U\).
\end{lemma}

\begin{proof}
Suppose $B \subseteq U^c$. Then after firing $U$, each vertex in $B$ has at least as many chips as it did prior to firing $U$. Since $B$ $r$-intersects $\text{msupp}(D)$, it must be that $B$ $r$-intersects $\text{msupp}(D')$, a contradiction. 

Thus we know $B$ is not contained in $U^c$. If $|B| = 1$, then $B \subseteq U$ and we are done. Otherwise, suppose  $B \cap U \neq \emptyset$ and $B \cap U^c \neq \emptyset$. Since $B$ is $r$-edge connected, there must be at least $r$ edges in  $E(B \cap U,B \cap U^c)$. Since a chip travels along each edge in this cut set, there will be at least $r$ chips in $B \cap U^c$ after firing $U$. Thus, we have $|B \cap \text{msupp}(D')| \geq r$, a contradiction. Therefore, it must be that $B \cap U^c = \emptyset$, and so $B \subseteq U$. 
\end{proof}

Given a graph $G$, an \emph{$r$-scramble} $\mathcal{S}$ on $G$ is a collection of subsets \(E_i\subset V(G)\) such tht \(G[E_i]\) is $r$-edge connected (with \(G[E_i]\) a single vertex permitted).  The \emph{$r$-hitting number of $\mathcal{S}$}, denoted \(h_r(\mathcal{S})\), is the minimum size of a multisubset \(C\) of $V(G)$ such that $C$ \(r\)-intersects every egg \(E_i\in \mathcal{S}\). The egg-cut number of \(\cS\), denoted \(e(\cS)\), is defined the same as for usual scramble number. We then define the \emph{order} of an $r$-scramble to be the minimum of \(h_r(\cS)\) and \(e(\cS)\):
\[||\cS||_r=\min\{h_r(\cS),e(\cS)\}.\]
The \emph{$r^{\text{th}}$ scramble number} of $G$, denoted $\sn_r(G)$, is the largest order of any $r$-scramble on $G$.  We remark that $\sn(G) = \sn_1(G)$.

\begin{theorem} \label{thm_higher_scrambles}
For any graph $G$ and any integer $r \geq 1$, we have $\sn_r(G) \leq \gon_r(G)$.
\end{theorem}

Our proof closely follows the proof of \cite[Theorem 4.1]{scramble}.

\begin{proof} 
Let $\cS$ be an $r$-scramble on $G$, and let $D'$ be an effective divisor on $G$ with $r(D') \geq r$. We will show that $\deg(D') \geq ||\cS||_r$.  Among the effective divisors equivalent to $D'$, choose a divisor $D$ such that $\text{msupp}(D)$ $r$-intersects the maximum possible number of eggs in $\cS$. If $\text{msupp}(D)$ is an $r$-hitting set for $\cS$, then we have
$$\deg(D) = |\text{msupp}(D)|\geq h_r(\cS) \geq \text{min}\{h_r(\cS), e(\cS)\}=||\cS||_r.$$

So assume that there is some egg $E \in \cS$ that does not $r$-intersect $\text{msupp}(D)$, and let $v \in E$. We note that since the multiset $\text{msupp}(D)$ intersects $E$ in at most $r-1$ vertices, we have $D(v) \leq r-1$. And since $r(D) \geq r$, it follows that $D$ is not $v$-reduced, since an effective divisor with rank at least $r$ will have an effective $v$-reduced form with at least $r$ chips on $v$. By \cite[\S 3]{db}, we know there exists a chain of sets 
$$\emptyset \subsetneq U_1 \subseteq \dots \subseteq U_k \subseteq V(G) \setminus \{v\}$$
and a sequence of effective divisors $D_0, D_1, \dots, D_k$ such that:
\begin{itemize}
    \item$D_0 = D$,
    \item$D_k$ is $v$-reduced, and
    \item $D_i$ is obtained from $D_{i-1}$ by firing the set $U_i$, for all $i$.
\end{itemize}

Since $r(D)\geq r$, and $D_k$ is $v$-reduced, we have that $v$ appears at least $r$ times in $\text{msupp}(D_k)$. Hence $\text{msupp}(D_k)$ $r$-intersects $E$; that is, $|\text{msupp}(D_k) \cap E| \geq r$. By our maximality assumption of $D$, we know that $\text{msupp}(D_k)$ does not $r$-intersect more eggs than $\text{msupp}(D)$, and since $E$ $r$-intersects $\text{msupp}(D_k)$ but not $\text{msupp}(D)$, there must be at least one egg $E'$ that $r$-intersects $\text{msupp}(D)$ but not $\text{msupp}(D_k)$, i.e. $\text{msupp}(D_k)$ intersects $E$ on fewer than \(r\) vertices. Let $i \leq k$ be the smallest index such that there is some $E' \in \cS$ that $r$-intersects $\text{msupp}(D)$ but not $\text{msupp}(D_i)$. Then $|E' \cap \text{msupp}(D_{i-1})| \geq r$ and $|E' \cap \text{msupp}(D_{i})| < r$. By Lemma \ref{lem_egg_containment}, it follows that $E' \subseteq U_i$.

Again, by assumption, $\text{msupp}(D_i)$ does not $r$-intersect more eggs than $\text{msupp}(D)$, so $\text{msupp}(D_i)$ does not $r$-intersect $E$. 
Let $j \geq i$ be the smallest index such that $|E \cap \text{msupp}(D_{j-1})| < r$ and $|E \cap \text{msupp}(D_j)| \geq r$. Since $\text{msupp}(D_{j-1})$ can be obtained by $\text{msupp}(D_{j})$  by firing $U_j^c$, by Lemma \ref{lem_egg_containment}, we see that $E \subseteq U_j^c \subseteq U_i^c$. Since $E \subseteq U_i^c$ and $E' \subseteq U_i$, it follows from the definition of a $r$-scramble that $|E(U_i, U_i^c)| \geq e(\cS)\geq \text{min}\{h_r(\cS), e(\cS)\}=||\cS||_r$. Since
$$\deg (D_{i-1}) \geq \sum_{u \in U_i} D_{i-1}(u) \geq |E(U_i, U_i^c)|,$$
we have
$$\deg (D_{i-1}) \geq ||\cS||_r$$

In all cases we have $||\cS||_r \leq \deg(D)= \gon_r(G)$. Since \(\cS\) was an arbitrary \(r\)-scramble, we conclude that $\sn_r(G) \leq \gon_r(G)$.
\end{proof}

We can readily prove the following result.

\begin{corollary}
If $G$ is a graph on $n$ vertices with edge-connectivity $\lambda(G)$ at least $2n$, then $\sn_2(G) = \gon_2(G) = 2n$.
\end{corollary}
\begin{proof}
Let $\cS$ be the vertex scramble on a graph $G$, whose eggs are precisely the subsets of \(V(G)\) of size \(1\). Since $\lambda(G) \geq 2n$, the size of any egg cut must be at least $2n$. Furthermore, $h_2(S) = 2n$ since every vertex must be hit twice, as the eggs are pairwise disjoint. Thus, $||S||= 2n$, so $\sn_2(G) \geq 2n$. We also have that placing 2 chips on every vertex is a winning chip placement for second gonality, since subtracting two chips cannot introduce debt, so $\gon_2(G) \leq 2n$. By Theorem \ref{thm_higher_scrambles}, we have $\sn_2(G) = \gon_2(G) = 2n$. 
\end{proof}

This corollary applies to \emph{generalized banana graphs} with sufficiently many edges.  These graphs have a path graph as their underlying simple graph, with various numbers of parallel edges between adjacent vertices.  If a generalized banana graph \(G\) on \(n\) vertices has at least \(2n\) edges between each pair of adjacent vertices, then it satisfies all assumptions of the corollary, implying \(\gon_2(G)=2n\).  We remark that this was previously proved in \cite[Lemma 5.2]{gonseq} in the case that every multiedge had the same number of edges, but the proof was  more cumbersome and required more case-checking.

It turns out we can similarly generalize brambles and treewidth to define $r^{th}$ bramble number and $r^{th}$ treewidth. The main result we will prove on this front is the following theorem.

\begin{theorem}\label{theorem:tw_sn_gon}
For a graph \(G\), we have
\[\textrm{tw}_r(G)\leq \textrm{sn}_r(G)\leq \gon_r(G).\]
\end{theorem}

Although \(r^{th}\) treewidth provides a weaker lower bound than \(r^{th}\) scramble number, it is still a natural generalization, and one that may warrant future investigation. In particular, our framework does not include a generalization of tree decompositions, which would be an interesting direction for future study.

 We define an \emph{$r$-bramble} to be a set $\mathcal{B} = \{B_1, \dots, B_k\}$ of $r$-edge-connected subsets of $V(G)$ with the property that for any $B, B' \in \mathcal{B}$, either $B \cap B' \neq \emptyset$ or $B \cap B' = \emptyset$ and there are at least $r$ edges in $E(B, B')$. In the latter case, we say that $B$ and $B'$  \emph{$r$-touch}.  Note that the definition of a bramble is equivalent to the definition of a \(1\)-bramble.

The order of an $r$-bramble $\mathcal{B}$, denoted $||\mathcal{B}||_r$, is the minimum size of an $r$-hitting set for $\mathcal{B}$. That is, 
$$||\mathcal{B}||_r = h_r(\mathcal{B}).$$ The $r^{th}$ bramble number $\text{bn}_r(G)$ of a graph $G$ is then the maximum order of an $r$-bramble on $G$. That is,
$$\text{bn}_r(G) := \max_{\mathcal{B}} \{||\mathcal{B}||_r\},$$
where the maximum is taken over all \(r\)-brambles.

The \emph{$r^{th}$ treewidth} of a graph $G$, denoted $\tw_r(G)$, is  defined as \(r\) less than the \(r^{th}\) bramble number:
\[\tw_r(G)=\textrm{bn}_r(G)-r.\]

Note that \(\tw(G)=\tw_1(G)\).  To prove Theorem \ref{theorem:tw_sn_gon}, it will suffice to show that \(\textrm{tw}_r(G)\leq \textrm{sn}_r(G)\).  We start with the following lemma.

\begin{lemma}
Let $\mathcal{B}$ be an $r$-bramble in $G$ and let $U \subseteq V=V(G)$. Suppose that there exist $B, B' \in \mathcal{B}$ such that $B \subseteq V \setminus U$ and $B' \subseteq U$. Then $|E(U, V \setminus U)| + r \geq ||\mathcal{B}||_r$.
\end{lemma}

Our proof closely follows the proof of \cite[Lemma 2.3]{debruyn2014treewidth}.

\begin{proof}

We will construct a hitting set for $\mathcal{B}$ of size at most $|E(U, V \setminus U)|+ r$. Let $F := E(U, V \setminus U)$ be the cut determined by $U$ and let $H := (V, F)$. Let
$$X := \{v \in U ~ | ~ d_H(v) \geq 1 \} ~~~~~~ \text{ and } ~~~~~~~ Y := \{v \in V \setminus U ~ | ~ d_H(v) \geq 1 \} $$ be the `shores' of the cut $F$, where $d_H(v)$ is the valence of $v$ when only considering edges in $H$.

Let $\mathcal{B}' := \{B' \in \mathcal{B} | B' \subseteq U\}$. By assumption, $\mathcal{B}'$ is nonempty. Chose $B' \in \mathcal{B'}$ for which $B' \cap X$ is inclusionwise minimal. Let $B \in \mathcal{B}$ be such that $B \subseteq V \setminus U$. Since $B'$ must \(r\)-touch \(B\), \(B'\cap X\) must contain at least \(r\) elements (as a multiset). See Figure \ref{fig:higherbrambles} for an illustration.
\medskip
\begin{figure}[hbt]
    \centering
    \includegraphics[scale=.30]{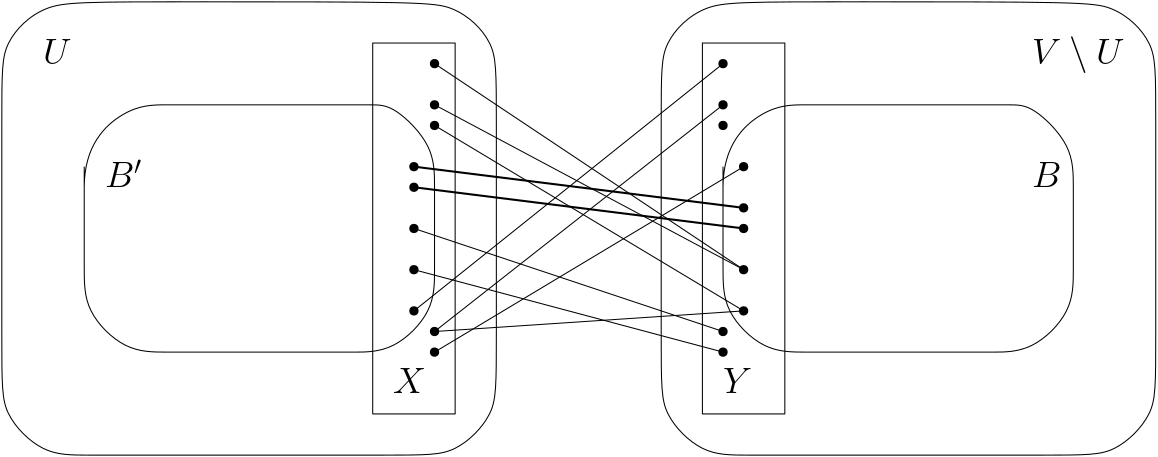}
    \caption{A visualization of the shores of a cut separating $B, B' \in \mathcal{B}$ for an $r^{th}$ order bramble $\mathcal{B}$. Here, $r = 2$. This image is modeled off of Figure 1 of \cite{debruyn2014treewidth}.}
    \label{fig:higherbrambles}
\end{figure}

We now construct an $r$-hitting multiset $S$ for $\mathcal{B}$.
\begin{enumerate}
    \item Add $r$ arbitrary elements $s_1, s_1, \dots, s_r$ from $B' \cap X$ to our multiset $S$. We allow $s_i = s_j$ (with $i \neq j$) to occur, and in this case add $s_i$ to the multiset $S$ twice.
    \item For each edge $xy \in E(X, Y)$ with $x \in X, y \in Y$, we add $x$ to $S$ if $x \not \in B'$, and otherwise we add $y$ to $S$. It is acceptable to add the same vertex to our multiset $S$ multiple times, again as $S$ is a multiset.
\end{enumerate}

By construction, the multiset \(S\) has $2r+ |F|$ elements.  We must now prove that $S$ is an $r$-hitting set for $\mathcal{B}$. 
Let $A \in \mathcal{B}$. 
We note that $A$ intersects $X \cup Y$; otherwise, it would fail to touch at least one of \(B\) and \(B'\). We now claim that \(G[A]\)  has at least $r$ edges in $E(U, V \setminus U)$. Suppose that this claim is false. There are two cases to deal with: that $A$ intersects exactly one of $X$ or $Y$ and has at most $r - 1$ edges in $E(U, V \setminus U)$; or that $A$ intersects both of $X$ and $Y$ and has at most $r - 1$ edges in $E(U, V \setminus U)$.

In the first case, without loss of generality assume $A \cap Y = \emptyset$. Then $A$ and $B$ are connected by up to only up to only up to $r - 1$ edges, meaning they do not \(r\)-touch, a contradiction. 

In the second case, $A$ intersects both of $X$ and $Y$ and has at most $r-1$ edges in $E(U, V \setminus U)$. By deleting these $\leq r-1$ edges, we would disconnect $A$ into two components: the component in $U$ and the component in $V \setminus U$. This contradicts the $r$-edge connectivity condition of each $A \in \mathcal{B}$.

In both cases, we reach a contradiction to $\mathcal{B}$ being an $r$-bramble.  Thus \(G[A]\) contains at least \(r\) edges from \(E(X,Y)\).

We now consider three cases:  that \(A\cap Y=\emptyset\), that \(A\cap X=\emptyset\), and that \(A\cap X\) and \(A\cap Y\) are both nonempty.

\begin{enumerate}
    \item Suppose $A \cap Y = \emptyset$. In this case, $A \subseteq U$. By the choice of $B'$, we have either $B' \cap X \subseteq A \cap X$ and hence $s_1, s_2,\dots, s_r \in A$, or there exists $r$ edges leaving $A$ that are in $E(U, V \setminus U)$. In the latter case, there exist $r$ elements $x_1, x_2, \dots, x_r \in (X \cap A) \setminus B'$. We note that $x_i = x_j$ for $i \neq j$ is allowed, if $x_i$ is the endpoint of $A$ for two edges leaving $A$, and in this case $x_i$ was added twice to the multiset $S$. Our construction of $S$ implies that $x_1, x_2, \dots, x_r \in S$. Thus in this case, $S$ \(r\)-hits $A$.
    \item Suppose $A \cap X = \emptyset$. In this case $A \subseteq (V \setminus U)$. Since $A$ has $r$ edges going to $B'$ (since $A$ and $B'$ are in disjoint components of the graph $G$ so they cannot intersect), there must be $r$ edges $e_i = x_iy_i$ for all $i$ from 1 to $r$, with all $x_i \in B' \cap X$ and all $y_i \in A \cap Y$. By construction of $S$ we have that all $y_i \in S$ for $i$ from 1 to $r$.
    \item Suppose $A \cap X \neq \emptyset$ and $A \cap Y \neq \emptyset$. Since $G[A]$ is $r$-edge-connected, there must be $r$ edges $e_i = x_iy_i$, \(1\leq i\leq r\), such that $x_i \in X$, $y_i \in Y$, and $x_i, y_i \in A$ for all \(i\). Since $S$ contains at least one endpoint from each edge in $F$, the set $S$ must hit $A$ $\geq r$ times.
\end{enumerate}

We conclude that $S$ is an $r$-hitting set for $\mathcal{B}$ of size at most $|E(U, V \setminus U)| + r$, which proves the claim.
\end{proof}

Since any \(r\)-bramble \(\mathcal{B}\) is also a scramble (indeed, an \(r\)-scramble), we may consider \(e(\mathcal{B})\) as usual.
From the previous lemma, it immediately follows that $e(\mathcal{B}) + r \geq ||\mathcal{B}||_r$.

\begin{proof}[Proof of Theorem \ref{theorem:tw_sn_gon}]
Let \(\mathcal{B}\) be an \(r\)-bramble on \(G\) with \(||\mathcal{B}||_r=h_r(\mathcal{B})=\tw_r(G)+r\).  Let \(\mathcal{S}=\mathcal{B}\) now considered as an \(r\)-scramble.  Note that \(e(\mathcal{S})\geq ||\mathcal{B}||_r-r=h_r(\mathcal{B})-r=\tw_r(G)\) and \(h_r(\mathcal{S})=h_r(\mathcal{B})=\tw_r(G)+r\).  Since both \(e(\mathcal{S})\) and \(h_r(\mathcal{S})\) are at least \(\tw_r(G)\), we have that
\[\sn_r(G)\geq ||\mathcal{S}||_r\geq\tw_r(G),\]
as desired.
\end{proof}

\bibliographystyle{plain}

\end{document}